\documentclass[12pt,twoside,reqno]{amsart}
\usepackage{amsmath}
\usepackage{amssymb}
\usepackage{wasysym}
\usepackage{psfrag}
\usepackage{graphicx,epsf,amsmath}  
\usepackage{epsf,graphicx}
\usepackage{hyperref}

\usepackage{marginnote}
\usepackage[draft]{changes}
\definechangesauthor[color=blue]{hx}
\newtheorem{lem}{Lemma}[section]
\newtheorem{theorem}[lem]{Theorem}

\newtheorem{defi}[lem]{Definition}

\newtheorem{rema}[lem]{Remark}

\numberwithin{equation}{section}

\def\epsilon{\varepsilon}

\DeclareMathOperator{\Vol}{Vol}
\DeclareMathOperator{\supp}{supp}

\title[]
{On a new functional for extremal metrics of the conformal laplacian in high dimension }

\begin{document}
\begin{abstract}
In this paper, we introduce a new functional for the conformal spectrum of the conformal laplacian on a closed manifold $M$ of dimension at least $3$. For this new functional we provide a Korevaar type result. The main body of the paper deals with the case of the sphere but a section is devoted to more general closed manifolds. 
\end{abstract}


\author [Sire]{Yannick Sire}
\address{Department of Mathematics, Johns Hopkins University, Baltimore, MA 21218, USA} \email{\href{mailto:sire@math.jhu.edu}{sire@math.jhu.edu}}

\author [Xu]{Hang Xu}
\address{Department of Mathematics, University of California, San Diego, 9500 Gilman Dr, La Jolla, CA 92093 , USA}
\email{\href{mailto:h9xu@ucsd.edu}{h9xu@ucsd.edu }}

\maketitle
\tableofcontents

\section{Introduction}
Let $(M^n,g)$ with $n \geq 3$ be a compact boundaryless manifold with scalar curvature $R_g$. Let $[g]$ be the conformal class of $g$ and let $c_n$ denote the constant $\frac{n-2}{4(n-1)}$. Consider the conformal laplacian 
\begin{equation*}
	\Box_g=-\Delta_g +c_nR_g.
\end{equation*}

Let $\widetilde g$ be a conformal metric to $g$,  i.e. $\widetilde g= \mu^{4/(n-2)} g$ for some positive function $\mu\in C^{\infty}(M)$. We investigate the eigenvalue problem 
\begin{equation}\label{eigen}
(-\Delta_{\widetilde g}+c_nR_{\widetilde g})u=\lambda u 
\end{equation}

Suppose $(S^n,g)$ is the $n$-sphere with the round metric and $\widetilde{g}\in [g]$.
Recall that a celebrated result by Korevaar \cite{Korevaarupperbounds} implies that the $k$th eigenvalue $\lambda_k(S^n,-\Delta_{\widetilde{g}})$ of the laplacian satisfies
\begin{equation}\label{Korevaar}
\lambda_k(S^n,-\Delta_{\widetilde{g}})\cdot \Vol(S^n,\widetilde{g})^{2/n}\leq C(n)k^{2/n}.
\end{equation}
However, the same estimate is false for the conformal laplacian \cite{conformaleigenvalueunbounded}. By changing the quantity $\Vol(S^n,\widetilde{g})^{2/n}=\left(\int_{S^n}\mu^{\frac{2n}{n-2}}dV_g\right)^{2/n}$ to a smaller one in the sense of Lebesgue space embeddings, we have 
\begin{theorem}\label{main}
	Let $(S^n,g)$ be the $n$-sphere with the round metric. For any metric $\widetilde{g}\in [g]$, the $k$th eigenvalue  of the conformal laplacian $\Box_{\widetilde{g}}$ satisfies the inequality
	\begin{equation*}
	\lambda_k(S^n,\Box_{\widetilde{g}})\int_{S^n}\mu^{\frac{4}{n-2}}dV_g\leq C(n)k^{2/n},
	\end{equation*}
	where $C(n)$ is a constant only depending on $n$.
\end{theorem}

In order to put the previous result in perspective, we introduce the quantity
$$
\Lambda_{k,p}(\widetilde{g})=\lambda_k(S^n, \Box_{\widetilde{g}}) \,\|\widetilde{g}\|_p,
$$
where 
$$\|\widetilde{g}\|_p = \Big ( \int_{S^n} \mu ^{\frac{4p}{n-2}}  dV_g\Big )^{1/p}$$ 
for $p \geq 1.$

Therefore, \eqref{Korevaar} can be viewed as {\sl universal bounds} on $\Lambda_{k,n/2}(\widetilde{g})$ for the {\sl Laplace} eigenvalues. Our result, for the conformal Laplacian eigenvalues, can then be formulated as 
$$
\Lambda_{k,1}(\widetilde{g}) \leq C(n) k^{2/n}.
$$
We would like to emphasize that the functional under consideration in the previous Theorem is {\sl not} invariant under a conformal change of the metric and appears to be new in the literature. Despite its lack of conformal invariance, in Section \ref{critical point}, we show that (smooth) critical points of  $\Lambda_{k,1}(\widetilde{g})$ satisfy an Euler-Lagrange equation, which is, once suitably formulated, completely analogous to the one of appearing in the surface case. This aspect, which allows also to give a PDE insight on the counter-example of Amman and Jammes, shows that the functional under consideration {\sl is} in fact the natural one for higher dimensions. 

The proof of Theorem \ref{main} follows the approach from \cite{KokarevConformal,KokarevKahler}, in which the idea traces back to \cite{annuluscovering} and \cite{Korevaarupperbounds}, i.e. by considering suitable test functions on annuli for the functional under consideration. Using more directly the approach of Korevaar in \cite{Korevaarupperbounds} we believe we could also obtain the desired upper bounds by a simpler argument. However, the approach we took is somehow more general as the work of Kokarev deals with test functions for eigenvalues of the laplacian on K\"ahler manifolds in a fixed K\"ahler class and we believe that, working it out in a simpler Riemannian case like the one we consider, could be useful to deal with other types of spectral functionals. 

\subsection*{Some remarks about the existence or no-existence of extremal metrics} 
Since the spectral functional $\Lambda_{k,1}(\widetilde{g})$ is bounded above universally it is a very natural question to consider its supremum on all possible conformal metrics. More generally, one could consider a family of functionals $\Lambda_{k,p}(\widetilde{g})$ with $1 \leq p \leq \frac{n}{2} $ for either the Laplacian or the conformal laplacian. We briefly describe below the state of the art for several of these functionals. 
In the case that of the spectral functional $\Lambda_{k,n/2}(\widetilde{g})$ for the Laplacian, the case for $n=2$ is completely understood thanks to the recent important work \cite{KNPP} (see also \cite{N3}). For $n \geq 3$, it is known that the first eigenvalue as a function fo the metric is unbounded if one does not fix a conformal class  by a result of Urakawa \cite{urakawa}. In the case $n=2$ and for surfaces different from the sphere, we refer the reader to \cite{N1,N2,petrides1,petrides2}. Whenever the case of the Conformal Laplacian is concerned, we already mentioned that the functional is unbounded by \cite{conformaleigenvalueunbounded}. Turning now to the case $1 \leq p<\frac{n}{2}$, the lack of the conformal invariance makes the problem of extremal metrics geometrically harder.  In Section \ref{critical point} , we derive equations of Schr\"odinger type which are the natural {\sl critical point} of the previous functional. The existence of an extremal metric for a given eigenvalue and its regularity is a hard problem (in particular the regularity aspects (see \cite{N1,N2,petrides1,petrides2}). To put our problem in a slightly wider context, one can consider the space $\mathcal M (S^n)$ the space of conformal structures on $S^n$ for $n \geq 3$. We recall also that there is only one conformal structure on $S^2$. Therefore, extremal metrics (and in particular suprema) are critical points of the eigenvalue functional defined on $\mathcal M(S^n)$ restricted to the $L^p$ unit sphere of $\mathcal M(S^n)$ for $1 \leq p <n/2$. Within this formulation one could think to apply min-max arguments (see \cite{KS} for related results).

Whenever the spectral functional is conformally invariant, extremal metrics and their associated eigen-elements constitute the conformal spectrum of the operator. It plays an important  role in the geometry of minimal submanifolds. An extensive program initiated by S. T. Yau and collaborators (see \cite{LiYau,YangYau} for instance) has been instrumental in these aspects. The existence of extremal metrics on general surfaces has been proved by Nadirashvili and the first author \cite{N1,N2} (see also \cite{nadir}). 

The formulation of the problem of extremal metrics for our functional as described in Section \ref{critical point} opens the possibility to investigate existence and regularity of extremal metrics as developed in \cite{N1,N2} for instance. Indeed, the strategy implemented in \cite{N1} uses the fact that the problem is two-dimensional only at the very last step. This part of the program is under investigation \cite{SX}.

\section{Critical points of the spectral functional}\label{critical point}
We derive the Euler-Lagrange equation of the spectral functional under consideration on the space of conformal metrics. The aim is to have a better insight on the PDE aspects of the problem. 
Let $(M^n,g)$ with $n \geq 3$ be a closed manifold with scalar curvature $R_g$. Let $[g]$ be the conformal class of $g$ and consider the conformal laplacian for $\widetilde g \in [g]$
$$
\Box_{\widetilde g}=-c_n\Delta_{\widetilde g} +R_{\widetilde g}
$$ 

We set $\tilde g= \mu^{\frac{4}{n-2}} g$ and we consider the eigenvalue problem 
\begin{equation}\label{eigenApp}
(-\Delta_{\tilde g}+R_{\tilde g})u=\lambda u 
\end{equation}
We have 
\begin{lem}
The eigenvalue problem \eqref{eigen} is equivalent to the following eigenvalue problem 
\begin{equation}\label{eigenConf}
(-\Delta_{g}+R_{ g})v=\lambda \mu^\frac{4}{n-2} v 
\end{equation}
\end{lem}
\begin{proof}
The conformal law for the conformal laplacian gives for any $f \in C^\infty(M)$ (say)  and metrics related by $\tilde g= \mu^\frac{4}{n-2} g$
$$
(-\Delta_{g}+R_{g})(\mu \, f)=\mu^\frac{n+2}{n-2}\, (-\Delta_{\tilde g}+R_{\tilde g})(f). 
$$
Now take $f \equiv u$ where $u$ satisfies \eqref{eigenApp}. This gives 
$$
(-\Delta_{g}+R_{g})(\mu \, u)=\lambda \mu^\frac{n+2}{n-2}\, u
$$
Setting $v=\mu\, u$ gives the result. 
\end{proof}
The spectral problem in Theorem \ref{main} then amounts to solve in $(\lambda, v, \mu)$ the following system
\begin{equation}\label{crit}
\left \{
\begin{array}{c}
(-\Delta_{g}+R_{ g})v=\lambda \mu^\frac{4}{n-2} v,\\
\int_{M} \mu^\frac{4}{n-2}dV_g=1
\end{array} \right .
\end{equation}

Introducing the new potential $\mathcal V(x)=\mu^\frac{4}{n-2}(x)$, we end up with the system 
\begin{equation}\label{crit2}
\left \{
\begin{array}{c}
(-\Delta_{g}+R_{ g})v=\lambda \mathcal V v,\\
\int_{M} \mathcal VdV_g=1
\end{array} \right .
\end{equation}
This system is completely analogous to the one on surfaces derived in \cite{N1}.  On the other hand, in the case of the Ammann-Jammes functional \cite{conformaleigenvalueunbounded}, i.e. whenever one normalizes by the volume, one obtains 
\begin{equation}\label{pbAmmann}
\left \{
\begin{array}{c}
(-\Delta_{g}+R_{ g})\, v=\lambda \mathcal V(x) v\,\,\,\mbox{on}\,\,M \\
\int_M \mathcal V^{n/2}dv_g=1.  
\end{array}\right .
\end{equation}

Notice that in this case the potential $\mathcal V$ giving rise to the extremal metric appears as a {\sl critical } norm $L^{n/2}$ constraint in dimension $n$. However, with the normalization we choose, the constraint is {\sl subcritical}.

\section{Geometry of metric measure space}\label{section metric space}
In this section, we will show the existence a large number of disjoint annuli carrying a sufficient amount of the volume measure. 
Let $(X,d)$ be a metric space. For any $p\in X$ and $0\leq r<R<+\infty$, we denote the ball $B(p,r)=\{x\in X: d(x,p)<r\}$ and the annulus $A(p;r,R)=\{x\in X: r\leq d(x,p)<R\}$. Recall the following definition.
\begin{defi}\label{N-covering}
	Given a positive integer $N$, we say that a metric space $(X,d)$ satisfies $N$-covering property if for any ball $B(p,r)$ in $X$, there exists a family of at most $N$ balls of radii $r/2$, which cover $B(p,r)$.
\end{defi}

The main tool we will use to construct disjoint annuli is the following theorem by Grigoryan, Netrusov and Yau \cite{annuluscovering} (see Theorem 3.5 and Lemma 3.10).
\begin{theorem}[\cite{annuluscovering}]\label{covering theorem}
	Let $(X,d)$ be a metric space and $m$ be a Borel measure. Assume that
	\begin{itemize}
		\item $(X,d)$ satisfies $N$-covering property;
		\item all balls in $X$ are precompact;
		\item measure $m$ is finite and non-atomic.
	\end{itemize} 
	Then for any positive integer $k$, there exists annuli $A_j=A(p_j; r_j, R_j)$ for $1\leq j\leq k$ such that
	\begin{itemize}
		\item $2A_j=A_j(p_j; \frac{r_j}{2}, 2R_j)$ for $1\leq j \leq k$ are pairly disjoint;
		\item for each $1\leq j\leq k$, $m(A_j)\geq c\frac{m(X)}{k}$. 
	\end{itemize} 
	Here $c=c(N)$ is a constant only depending on $N$.
\end{theorem}

In our case, we will set metric space $(X, d)=(S^n, d_g)$, where $d_g$ is the distance function induced from the round metric $g$. Since $(S^n,g)$ is positively curved, it satisfies the \emph{doubling property}. That is to say, there exists some constant $C(n)$, such that $\Vol_g(B(p,2r))\leq C(n)\Vol_g(B(p,r))$ for any ball $B(p,r)\subset S^n$. In fact, we can take $C(n)=2^n$ for $(S^n,d_g)$ by using the volume comparison theorem. And it follows that $(S^n,d_g)$ satisfies $N$-covering property for any $N\geq 4^n$ by using a packing argument. In particular, the constant $c$ here only depends on the dimension $n$ if we set $N=4^n$.

Let $\nu, \widetilde{\nu}$ be the volume measure induced by the metric $g,\widetilde{g}$ respectively. Furthermore, we will set the Borel measure 
\begin{equation}
	m=\mu^{-2}d\widetilde{\nu}=\mu^{\frac{4}{n-2}}d\nu.
\end{equation}
Applying Theorem \ref{covering theorem} to $(S^n,d_g,\mu^{\frac{4}{n-2}}d\nu)$ and positive integer $2k$, we obtain a collection of annuli $\{A_j=A(p_j; r_j, R_j)\}_{j=1}^{2k}$ such that
\begin{enumerate}\label{covering property}
	\item[(i)] $2A_j=A_j(p_j; \frac{r_j}{2}, 2R_j)$ for $1\leq j \leq 2k$ are pairly disjoint;
	\item[(ii)] for each $1\leq j\leq 2k$, 
	\begin{equation}\label{covering property equation}
		\int_{A_j}\mu^{\frac{4}{n-2}}d\nu\geq \frac{c}{k}\int_{S^n}\mu^{\frac{4}{n-2}}d\nu=\frac{c}{k}m(S^n),
	\end{equation}
	where $c=c(n)$ is a constant only depending on $n$.
\end{enumerate}

Furthermore, we can assume the first $k$ many annuli in the collection $\{A_j\}_{j=1}^{2k}$ satisfy 
\begin{equation}\label{annulus measure upper bound}
{\nu}(2A_j)\leq \frac{{\nu}(S^n)}{k}, \quad \mbox{for } 1\leq j\leq k.
\end{equation}
This is because $\{2A_j\}_{j=1}^{2k}$ are pairly disjoint and thus we have
\begin{equation*}
\sum_{j=1}^{2k} {\nu}(2A_j)\leq {\nu}(S^n). 
\end{equation*}
By re-indexing the collection of annuli, we can assume ${\nu}(2A_1)\leq {\nu}(2A_2)\leq \cdots \leq {\nu}(2A_{2k})$ and $\eqref{annulus measure upper bound}$ clearly follows.

\section{construction of test functions}
In this section, we will construct test functions supported in each annulus $2A_j$ for $1\leq j\leq k$.

Fix an annulus $A=A(p; r, R)\subset S^n$. Let $x=(x^0,x^1,\cdots,x^{n})\in \mathbb{R}^{n+1}$ be the Cartesian coordinates. Let $\sigma_{p}: S^n\setminus\{p\}\rightarrow \mathbb{R}^n$ be the stereographic projection with the base point $p$ to the subspace
\begin{equation*}
	L_p=\{x\in \mathbb{R}^{n+1}: x\cdot p=0\}.
\end{equation*}
For any $t>0$, we denote $\delta_t:\mathbb{R}^n\rightarrow\mathbb{R}^n$ as the rescaling map by factor $t$, i.e., $\delta_t(y)=ty$ for any $y\in \mathbb{R}^n$. Define
\begin{equation*}
\theta_{p,t}=\sigma_{p}^{-1}\circ \delta_{t} \circ \sigma_{p},
\end{equation*}   
which is a conformal diffeomorphism of $(S^n,g)$, since each map in the composition is conformal. It is not hard to see the following properties of $\theta_{p,t}$ are satisfied.
\begin{itemize}
	\item $\theta_{p,t}$ fixes the points $\pm p$.
	\item $\theta_{p,t}$ preserves level sets of the distance function $d_g(p,\cdot)$.
	\item For any $x\in S^n\setminus\{\pm p\}$, $\theta_{p,t}(x)\rightarrow p$ as $t\rightarrow+\infty$, while $\theta_{p,t}(x)\rightarrow -p$ as $t\rightarrow 0$.
\end{itemize}

We start the construction with a special case that $r=0$ and $A(p;r,R)=B(p,R)$. Given $R\in (0,\frac{\pi}{2})$, we can choose the value $t=t(R)\in \mathbb{R}^+$ so that $\theta_{p,t}$ maps $B(p,2R)$ to the hemisphere $B(p,\frac{\pi}{2})$.
Let $x_p=x\cdot p$. We define
\begin{equation*}
	\varphi_{p,R}(x)=
	\begin{cases}
	x_p\circ \theta_{p,t}(x), & \quad \mbox{ if } x\in B(p,2R),\\
	0, & \quad \mbox{ if } x\not\in B(p,2R).
	\end{cases}
\end{equation*}
Let $\nabla$ be the Levi-Civita connection for the round metric $g$. Then we have the following lemma on the function $\varphi_{p,R}$.
\begin{lem}\label{test function on ball}
	$\varphi_{p,R}$ is a Lipschitz function on $S^n$ satisfying
\begin{itemize}
	\item $0\leq \varphi_{p,R}(x)\leq 1$ for any $x\in S^n$;
	\item $\varphi_{p,R}(x)\geq \frac{3}{5}$ for any $x\in B(p,R)$;
	\item $\supp \varphi_{p,R}\subset B(p,2R)$;
	\item 
	\begin{equation*}
		\int_{S^n}|\nabla \varphi_{p,R}|^n_gd\nu\leq C(n),
	\end{equation*}
	where $C(n)$ is a constant only depending on $n$.
\end{itemize}
\end{lem}
\begin{proof}
	The first property follows immediately by noting $x_p=x\cdot p=\cos\left(d_g(p,x)\right)$. The third property is also straightforward by the definition of $\varphi_{p,R}$. It remains to check the second one and the last one.
	
	We need to compute $t$ in terms of $R$ explicitly. As $(S^n,g)$ is homogeneous under the action of $\text{SO}(n+1)$, we can assume that $p=(1,0,0,\cdots,0)$. We denote $x'=(x^1,x^2,\cdots,x^n)$ and $x=(x^0,x')$. Recall the formulas for the stereographic projection:
	\begin{align*}
		\sigma_p(x)=\frac{x'}{1-x^0} \quad \mbox{ for any } x\in S^n, 
	\end{align*}
	and
	\begin{align*}
		\sigma_p^{-1}(y)=\left(\frac{|y|^2-1}{|y|^2+1},\frac{2y}{|y|^2+1}\right) \quad \mbox{ for any } y\in \mathbb{R}^n.
	\end{align*}
	By a straightforward computation, we obtain that for any $x\in S^n$
	\begin{equation}\label{conformal map theta}
		\theta_{p,t}(x)=\left(\frac{t^2(1+x^0)-(1-x^0)}{t^2(1+x^0)+(1-x^0)},\frac{2tx'}{t^2(1+x^0)+(1-x^0)}\right).
	\end{equation}
	Recall that $t$ is chosen so that $\theta_{p,t}$ maps $B(p,2R)$ onto $B(p,\frac{\pi}{2})$. In particular, $\theta_{p,t}$ maps $\partial B(p,2R)$ to $\partial B(p,\frac{\pi}{2})$ since it preserves level sets of the distance function $d_g(p,\cdot)$. Take $x\in \partial B(p,2R)$. Then 
	\begin{equation*}
		x_p=x^0=\cos(d_g(x,p))=\cos(2R).
	\end{equation*}
	And $\theta_{p,t}(x)\in \partial B(p,\frac{\pi}{2})$ implies that
	\begin{equation*}
		x_p(\theta_{p,t}(x))=\frac{t^2(1+x^0)-(1-x^0)}{t^2(1+x^0)+(1-x^0)}=\cos\left(\pi/2\right)=0.
	\end{equation*}
	Combine these two equations and we can solve
	\begin{equation*}
		t=\sqrt{\frac{1-\cos(2R)}{1+\cos(2R)}}=\tan R.
	\end{equation*}
	
	We are now ready to prove the second property. Since $\varphi_{p,R}(x)=\cos(d_g(p,\theta_{p,t}(x)))$ on the closed ball $x\in \overline{B(p,R)}$, the minimum is attained on the boundary $\partial B(p,R)$. For any $x\in \partial B(p,R)$, we have $x^0=\cos R$ and 
	\begin{equation*}
		\varphi_{p,R}(x)=\frac{\tan^2R\cdot(1+\cos R)-(1-\cos R)}{\tan^2 R\cdot (1+\cos R)+(1-\cos R)}=\frac{2\cos R+1}{2\cos^2R+2\cos R+1}.
	\end{equation*}
	We set $f(R)=\frac{2\cos R+1}{2\cos^2R+2\cos R+1}$. A straightforward computation shows that $f(R)$ is increasing on $(0,\frac{\pi}{2})$. Therefore, the second property follows by the fact
	\begin{equation*}
		\lim_{R\rightarrow 0}f(R)=\frac{3}{5}.
	\end{equation*}
	
	Next we prove the last property of $\varphi_{p,R}$. 
	Using the spherical coordinates, we have
	\begin{align*}
	&x^0=\cos\phi_1,\\
	&x^1=\sin\phi_1\cos\phi_2,\\
	&x^2=\sin\phi_1\sin\phi_2\cos\phi_3,\\
	&\cdots\\
	&x^{n-1}=\sin\phi_1\sin\phi_2\cdots\sin\phi_{n-1}\cos\phi_{n},\\
	&x^{n}=\sin\phi_1\sin\phi_2\cdots\sin\phi_{n-1}\sin\phi_{n}.
	\end{align*}
	Accordingly, the round metric writes into
	\begin{equation*}
	g=d\phi_1^2+\sin^2\phi_1d\phi_2^2+\sin^2\phi_1\sin^2\phi_2d\phi_3^2+\cdots+\sin^2\phi_1\sin^2\phi_2\cdots\sin^2\phi_{n-1}d\phi_{n}^2.
	\end{equation*} 
	Thus, using the expression in \eqref{conformal map theta}, we obtain
	\begin{align*}
	\int_{S^n}|\nabla\varphi_{p,R}|^nd\nu\leq &4^n\int_0^{\pi}\frac{t^{2n}\sin^{2n-1}\phi_1}{\left(t^2(1+\cos\phi_1)+1-\cos\phi_1\right)^{2n}}d\phi_1
	\\&
	\cdot
	\int_0^{\pi}\cdots\int_0^{\pi}\int_0^{2\pi}\sin^{n-2}\phi_2\sin^{n-3}\phi_3\cdots \sin\phi_{n-1} d\phi_2d\phi_3\cdots d\phi_{n}. 
	\end{align*}
	As the second integral on the right-hand side is a constant only depending on $n$, we just need to estimate the first one. Note that $t^2(1+\cos\phi_1)+1-\cos\phi_1\geq 2t\sin\phi_1$ by Cauchy-Schwarz inequality. It follows that
	\begin{align*}
	\int_0^{\pi}\frac{t^{2n}\sin^{2n-1}\phi_1}{\left(t^2(1+\cos\phi_1)+1-\cos\phi_1\right)^{2n}}d\phi_1
	\leq& 4^{1-n}\int_0^{\pi}\frac{t^{2}\sin\phi_1}{\left(t^2(1+\cos\phi_1)+1-\cos\phi_1\right)^2}d\phi_1\\
	=&4^{1-n}\int_{-1}^1\frac{t^{2}}{\left(t^2(1+x_1)+1-x_1\right)^2}dx_1\\
	=&2^{1-2n}.
	\end{align*}
	Therefore, there exists some constant $C(n)$ only depending on $n$, such that
	\begin{equation*}
	\int_{S^n}|\nabla\varphi_j|^nd\nu\leq C(n).
	\end{equation*}
\end{proof}

Similarly, for a given point $p\in S^n$ and $r\in (0,\pi)$, we choose $\tau=\tau(r)\in \mathbb{R}^+$ such that $\theta_{p,\tau}$ maps $B(p,\frac{r}{2})$ onto $B(p,\frac{\pi}{2})$. And we define
\begin{equation*}
\bar{\varphi}_{p,r}(x)=
\begin{cases}
0, & \quad \mbox{ if } x\in B(p,\frac{r}{2}),\\
-x_p\circ \theta_{p,\tau}(x), & \quad \mbox{ if } x\not\in B(p,\frac{r}{2}).
\end{cases}
\end{equation*}
$\bar{\varphi}_{p,r}(x)$ has similar properties as in Lemma \ref{test function on ball}.
\begin{lem}\label{test function outside ball}
	$\bar{\varphi}_{p,r}$ is a Lipschitz function on $S^n$ satisfying
	\begin{itemize}
		\item $0\leq \bar{\varphi}_{p,r}(x)\leq 1$ for any $x\in S^n$;
		\item $\bar{\varphi}_{p,r}(x)\geq \frac{3}{5}$ for any $x\not\in B(p,r)$;
		\item $\supp \bar{\varphi}_{p,r}\subset S^n\setminus B(p,\frac{r}{2})$.
		\item 
		\begin{equation*}
		\int_{S^n}|\nabla \bar{\varphi}_{p,r}|^n_gd\nu\leq C(n),
		\end{equation*}
		where $C(n)$ is a constant only depending on $n$.
	\end{itemize}
\end{lem}

Generally, for a given annulus $A(p;r,R)\subset S^n$, we construct the test function as 
\begin{equation*}
	\varphi_{p,r,R}=\varphi_{p,R}\cdot \bar{\varphi}_{p,r}.
\end{equation*}
Combining Lemma \ref{test function on ball} and Lemma \ref{test function outside ball}, we have
\begin{lem}\label{test function annulus}
	${\varphi}_{p,r,R}$ is a Lipschitz function on $S^n$ satisfying
	\begin{itemize}
		\item $0\leq \varphi_{p,r,R}(x)\leq 1$ for any $x\in S^n$;
		\item $\varphi_{p,r,R}(x)\geq \frac{9}{25}$ for any $x\in A(p;r,R)$;
		\item $\supp \varphi_{p,r,R}\subset A(p;\frac{r}{2},2R)$.
		\item 
		\begin{equation}\label{conformally invariant factor}
		\int_{S^n}|\nabla \varphi_{p,r,R}|^n_gd\nu\leq C(n),
		\end{equation}
		where $C(n)$ is a constant only depending on $n$.
	\end{itemize}
\end{lem}

\section{Proof of Theorem \ref{main}}
\begin{proof}[Proof of Theorem \ref{main}]
	
Recall the collection of annuli $\{A_j=A(p_j;r_j,R_j)\}_{j=1}^k$ satisfying properties (i), (ii) and \eqref{annulus measure upper bound} in Section \ref{section metric space}. 
For each annulus $A_j$ with $1\leq j\leq k$, set $\varphi_{j}:=\varphi_{p_j;r_j,R_j}$. We use $\nabla$, $\widetilde{\nabla}$ to denote the Levi-Civita connection for the metric $g$ and $\widetilde{g}$ respectively. For any Lipschitz function $f$ on $S^n$, let $\mathcal{R}(f)$ be the Rayleigh quotient:
\begin{equation*}
\mathcal{R}(f)=\frac{\int_{S^n}|\widetilde{\nabla} f|_{\widetilde{g}}^2+c_nR_{\widetilde{g}}|f|^2 d\widetilde{\nu}}{\int_{S^n}|f|^2d\widetilde{\nu}},
\end{equation*}
where $c_n=\frac{n-2}{4(n-1)}$ is the constant appearing in the conformal laplacian.

To prove Theorem \ref{main}, it is sufficient to show that for $1\leq j\leq k$, 
\begin{equation}
\mathcal{R}\left(\varphi_j/\mu\right)\leq C(n)k^{2/n} \left(\int_{S^n}\mu^{\frac{4}{n-2}}dV_g\right)^{-1}.
\end{equation}

We will estimate the numerator and denominator in the Rayleigh quotient respectively. 

We begin with the denominator. Since $\varphi_j\geq \frac{9}{25}$ on $A_j$ and by \eqref{covering property equation}, we have
\begin{equation}\label{denominator}
\int_{S^n} \left|\varphi_j/\mu\right|^2 d\widetilde{\nu}\geq \left(\frac{3}{5}\right)^4\int_{A_j}\mu^{-2}d\widetilde{\nu}\geq c\left(\frac{3}{5}\right)^4\frac{m(S^n)}{k},
\end{equation}
where $c$ is a constant only depends on $n$.

Now we consider the denominator. We can actually change the metric $\widetilde{g}$ to the round metric $g$ by the following lemma.
\begin{lem}
For any Lipschitz function $f$ on $S^n$, we have
\begin{equation}\label{integral under conformal change}
	\int_{S^n}|\widetilde{\nabla}  f|_{\widetilde{g}}^2+c_nR_{\widetilde{g}}|f|^2 d\widetilde{\nu}=\int_{S^n}|\nabla (\mu f)|_{g}^2+c_nR_{g}|\mu f|^2 d\nu
\end{equation}
\end{lem}
\begin{proof}
	It is sufficient to consider $f\in C^{\infty}(M)$. The conformal law for the conformal laplacian gives for metrics related by $\widetilde g= \mu^{4/(n-2)} g$
\begin{equation}\label{laplace under conformal change}
	(-\Delta_{g}+c_nR_{g})(\mu \, f)=\mu^{(n+2)/(n-2)}\, (-\Delta_{\tilde g}+c_nR_{\tilde g})(f)
\end{equation}
	And the volume measures are related by
	\begin{equation*}
		d\widetilde{\nu}=\mu^{2n/(n-2)}d\nu.
	\end{equation*}
	Therefore,
	\begin{align*}
		(-\Delta_{g}+c_nR_{g})(\mu \, f) \cdot (\mu f)d\nu
		=&\mu^{(n+2)/(n-2)}\, (-\Delta_{\tilde g}+c_nR_{\tilde g})(f)\cdot (\mu f)\mu^{-2n/(n-2)}d\widetilde{\nu}
		\\
		=&(-\Delta_{\tilde g}+c_nR_{\tilde g})(f)\cdot fd\widetilde{\nu}.
	\end{align*}
	The result follows immediately by integration by parts.
\end{proof}

Applying \eqref{integral under conformal change} to the numerator in $\mathcal{R}(\varphi_j/\mu)$, we obtain
\begin{equation}\label{numerator eq 1}
	\int_{S^n}|\widetilde{\nabla}  (\varphi_j/\mu)|_{\widetilde{g}}^2+c_nR_{\widetilde{g}}|\varphi_j/\mu|^2 d\widetilde{\nu}=\int_{S^n}|\nabla (\varphi_j)|_{g}^2+c_nR_{g}|\varphi_j|^2 d\nu.
\end{equation}
By using the H\"older inequality, as $\supp \varphi_j\subset 2A_j$,
\begin{equation*}
\int_{S^n}|\nabla\varphi_j|_g^2 d\nu
\leq \left(\int_{S^n}|\nabla\varphi_j|^n d\nu\right)^{2/n} \cdot \nu(2A_j)^{1-2/n}.
\end{equation*}
Recalling \eqref{annulus measure upper bound} and \eqref{conformally invariant factor}, it follows that
\begin{equation*}
\int_{S^n}|\nabla\varphi_j|_g^2 d\nu
\leq C(n)^{2/n} \cdot \left(\frac{\nu(S^n)}{k}\right)^{1-2/n}.
\end{equation*}
What's more, 
\begin{equation*}
	\int_{S^n} c_nR_g|\varphi_j|^2d\nu\leq n(n-1)c_n\cdot\nu(2A_j)
	\leq n(n-1)c_n\frac{\nu(S^n)}{k}. 
\end{equation*}
Plugging the above two inequalities into \eqref{numerator eq 1} and enlarging the constant $C(n)$ if necessary,
\begin{equation}\label{numerator}
	\int_{S^n}|\widetilde{\nabla}  (\varphi_j/\mu)|_{\widetilde{g}}^2+c_nR_{\widetilde{g}}|\varphi_j/\mu|^2 d\widetilde{\nu}
	\leq \frac{C(n)}{k^{1-2/n}}.
\end{equation}
Combining \eqref{denominator} and \eqref{numerator}, it follows that for any $1\leq j\leq k$,
\begin{equation*}
\mathcal{R}\left(\frac{\varphi_j}{\mu}\right)\leq \frac{C(n)}{c}\left(\frac{5}{3}\right)^4 k^{2/n}\left(m(S^n)\right)^{-1}.
\end{equation*}
The result therefore follows by renaming the constant $\frac{C(n)}{c}\left(\frac{5}{3}\right)^4$ to $C(n)$.
\end{proof}

\section{Generalizations and Related Questions}
\subsection{Generalizations to any closed manifolds}
Let $(M^n,g)$ be a closed Riemannian manifold of dimension $m\geq 3$. Then we can isometrically embed $(M,g)$ into $(S^N, g_0)$ for some $N\in \mathbb{N}$ depending on $n$ with the round metric $g_0$. Let $\phi: (M,g)\rightarrow (S^N,g_0)$ be such an embedding satisfying
\begin{equation*}
	\phi^*g_0=g.
\end{equation*}
Let $\widetilde{g}\in [g]$ be a conformal metric, which is related to $g$ as
\begin{equation*}
	\widetilde{g}=\mu^{4/(n-2)}g, \quad \mbox{ for some } \mu \in C^{\infty}(M).
\end{equation*} 
Recall the \emph{$N$-conformal volume $V_c(N,\phi)$ of $\phi$} (\cite{LiYau}), defined as 
\begin{equation}
	V_c(N,\phi)=\sup\{\Vol(M,(s\circ\phi)^*g_0): s \text{ is a conformal diffeomorphism of } S^N\}.
\end{equation}
For closed Riemannian manifolds, we can bound the eigenvalues of the conformal laplacians by the $N$-conformal volume $V_c(N,\phi)$.

\begin{theorem}\label{theorem submanifolds}
	For any metric $\widetilde{g}\in [g]$, the $k$th eigenvalue $\widetilde{\lambda}_k=\lambda_k(M^n,\widetilde{g})$ of the conformal laplacian $\Box_{\widetilde{g}}$ satisfies the inequality
	\begin{equation}\label{eigenvalue closed manifold}
		\widetilde{\lambda}_k \int_M \mu^{\frac{4}{n-2}} dV_{g}\leq C(n) \, \left(V_c(N,\phi) \, k^{2/n}+\sup_M R_g\right).
	\end{equation}
\end{theorem}
\begin{proof}
	Let $m=\phi_*\left(\mu^{4/(n-2)}dV_{g}\right)$ be the push-forward measure on $S^N$.
	For the metric space $(S^N, d_{g_0})$ and measure $m$ and any $k\in \mathbb{Z}^+$, we can similarly construct the collection of annuli $\{A_j\}_{j=1}^k$ satisfying properties (i), (ii) and \eqref{annulus measure upper bound} in Section \ref{section metric space}. And for each $A_j$, we can construct test functions $\varphi_j$ satisfying the first three properties as in Lemma \ref{test function annulus}. 
	
	We will consider the Rayleigh quotient $\mathcal{R}\left(\frac{\varphi_j\circ\phi}{\mu}\right)$. Using \eqref{covering property equation}, the denominator gives
	\begin{equation}\label{denominator general case}
		\int_{M} \left|\frac{\varphi_j\circ \phi}{\mu}\right|^2 dV_{\widetilde{g}}\geq \left(\frac{3}{5}\right)^4\int_{\phi^{-1}(A_j)}\mu^{-2}dV_{\widetilde{g}}\geq c\left(\frac{3}{5}\right)^4 \frac{m(S^N)}{k}.
	\end{equation}
	Let $\nabla$ and $\widetilde{\nabla}$ be the Levi-Civita connection for $g$ and $\widetilde{g}$ respectively. For the numerator of the Rayleigh quotient,
	\begin{equation}\label{numberator eq1}
		\int_{M}\left|\widetilde{\nabla}  \left(\frac{\varphi_j\circ \phi}{\mu}\right)\right|_{\widetilde{g}}^2+c_nR_{\widetilde{g}}\left|\frac{\varphi_j\circ\phi}{\mu}\right|^2 dV_{\widetilde{g}}=\int_{M}|\nabla (\varphi_j\circ\phi)|_{g}^2+c_nR_{g}|\varphi_j\circ \phi|^2 dV_{g}.
	\end{equation}
	The first term on the right side can be estimated as
	\begin{align*}
		\int_{M}|\nabla(\varphi_j\circ\phi)|_{g}^2 dV_{g}
		\leq \left(\int_{M}|\nabla(\varphi_j\circ\phi)|_{g}^n dV_{g}\right)^{2/n} \cdot \left(\Vol_{g}(2A_j)\right)^{1-2/n}.
	\end{align*}
	Recall that $\varphi_j=x_{p_j}\circ \theta_{p_j}$ for some conformal diffeomorphism $\theta_{p_j}$ of $S^N$. Since $\theta_{p_j}\circ\phi$ is conformal, we have
	\begin{equation*}
		\left(\theta_{p_j}\circ\phi\right)^*g_0=\frac{1}{n}\left|\nabla\left(\theta_{p_j}\circ\phi\right)\right|^2_g \,g.
	\end{equation*}
	Therefore,
	\begin{align*}
		\int_{M}|\nabla(\varphi_j\circ\phi)|_{g}^n dV_{g}
		\leq& \int_{M}|\nabla(\theta_{p_j}\circ\phi)|_{g}^n dV_{g}
		\\
		=& n^{n/2}\Vol(M,(\theta_{p_j}\circ\phi)^*g_0)
		\\
		\leq & n^{n/2} V_c(N,\phi).
	\end{align*}
	And by \eqref{annulus measure upper bound},
	\begin{equation*}
		\Vol_{g}(2A_j)\leq \frac{\Vol_{g}(M)}{k}
		\leq \frac{V_c(N,\phi)}{k}.
	\end{equation*}
	Therefore,
	\begin{align*}
		\int_{M}|\nabla(\varphi_j\circ\phi)|_{g}^2 dV_{g}
		\leq \frac{nV_c(N,\phi)}{k^{1-2/n}}.
	\end{align*}
	On the other hand, the second term on the right side of \eqref{numberator eq1} satisfies
	\begin{equation*}
		\int_M c_nR_{g}|\varphi_j\circ\phi|^2 dV_{g}
		\leq c_n\sup_{M} R_g \cdot\Vol_{g}(2A_j)\leq c_n\sup_M R_g\cdot \frac{V_c(N,\phi)}{k}.
	\end{equation*}
	Combining the above two estimates, we obtain
	\begin{equation}\label{numerator general case}
		\int_{M}\left|\widetilde{\nabla} \left(\frac{\varphi_j\circ \phi}{\mu}\right)\right|_{\widetilde{g}}^2+c_nR_{\widetilde{g}}\left|\frac{\varphi_j\circ\phi}{\mu}\right|^2 dV_{\widetilde{g}}
		\leq \left(n+c_n\frac{\sup_M R_g}{k^{2/n}}\right) \frac{V_c(N,\phi)}{k^{1-2/n}}.
	\end{equation}
	Using \eqref{numerator general case} and \eqref{denominator general case}, the desired estimates for the Rayleigh quotient $\mathcal{R}(\frac{\varphi_j\circ \phi}{\mu})$ is obtain and the result follows immediately.
\end{proof}

\begin{rema}
	Given an immersion $\phi: M\rightarrow S^N$, Li and Yau \cite{LiYau} gives the following condition for the conformal volume $V_c(N,\phi)$ to be identical to the volume of $M$.
	\begin{theorem}[\cite{LiYau}]
		Let $M$ be a homogeneous Riemannian manifold of dimension $n$. Suppose $\phi: M\rightarrow S^N$ is an immersion of $M$ into $S^N$ which satisfies the properties:
		\begin{itemize}
			\item[(i)] $\phi$ is an isometric minimal immersion.
			\item[(ii)] The transitive subgroup $H$ of the isometry group of $M$ is induced by a subgroup, also denoted by $H$, of $O(N+1)$ (i.e., $\phi$ is equivariant).
			\item[(iii)] $\phi(M)$ does not lie on any hyperplane of $\mathbb{R}^{N+1}$.
		\end{itemize}
		Then 
		\begin{equation*}
			V_c(N,\phi)=\Vol(M).
		\end{equation*}
	\end{theorem}
	In particular, when $M$ is an irreducible homogeneous manifold, a theorem of Takahashi (see \cite{Lawson}) says that one can minimally immerse $M$ isometrically into $S^N\subseteq \mathbb{R}^{N+1}$ by its first eigenspace of $M$. If we denote this isometric immersion $M\rightarrow S^N$ by $\phi$, then $V_c(N,\phi)=\Vol(M)$ and \eqref{eigenvalue closed manifold} writes into
	\begin{equation*}
	\widetilde{\lambda}_k \int_M \mu^{\frac{4}{n-2}} dV_{g}\leq C(n) \, \left(\Vol(M,g) \, k^{2/n}+\sup_M R_g\right).
	\end{equation*}
\end{rema}

\subsection{Hersch Type Results}
Let $(M,g)$ be a closed Riemannian manifold. For any metric $\widetilde{g}\in [g]$, related by $\widetilde{g}=\mu^{4/(n-2)}g$, we define the functional 
\begin{equation}
	\bar{\lambda}_k(M,g,\widetilde{g})=\left(\lambda_k(\widetilde g)\cdot \int_M \mu^{\frac{4}{n-2}}dV_g\right).
\end{equation}
And we also define the supreme over the conformal class as   
\begin{equation}
	\Lambda_k (M, g)=\sup_{\widetilde{g}\in [g]} \bar{\lambda}_k(M,g,\widetilde{g}).
\end{equation}
It is natural to ask what $\Lambda_k(M, g)$ is and whether $\Lambda_k(M, g)$ is achieved by certain Riemannian metric in general. If the maximal metric exist, it is defined up to multiplication by a positive constant due to the rescaling invariance of the functional. We can prove the following results in this direction.
\begin{theorem}
	For the sphere with the standard round metric $(S^n, g_0)$, 
	\begin{equation*}
		\Lambda_0(S^n, g_0)=\bar{\lambda}_0(S^n, g_0,g_0).
	\end{equation*}
	And $\Lambda_0(S^n, g_0)=\bar{\lambda}_0(S^n, g_0,\widetilde{g})$ holds only when $\widetilde{g}$ is the round metric up to scaling.
\end{theorem}
\begin{proof}
	For any $\widetilde{g}\in [g_0]$ related as $\widetilde{g}=\mu^{4/(n-2)}g_0$, take $1/\mu$ as the test function for the Rayleigh quotient and we obtain
	\begin{align}\label{equality holds}
		\lambda_0(S^n,\widetilde{g})\leq \frac{\int_{S^n} \left|\nabla_{\widetilde{g}} \left(\frac{1}{\mu}\right)\right|^2_{\widetilde{g}}+c_nR_{\widetilde{g}} \left|\frac{1}{\mu}\right|^2 dV_{\widetilde{g}}}{\int_{S^n} \left|\frac{1}{\mu}\right|^2dV_{\widetilde{g}}}.
	\end{align}
	By using \eqref{integral under conformal change}, the above inequality writes into
	\begin{align*}
		\lambda_0(S^n,\widetilde{g})\leq
		\frac{\int_{S^n}c_n R_{g_0}dV_{g_0}}{\int_{S^n} \mu^{4/(n-2)}dV_{g_0}}=\frac{\bar{\lambda}_0(S^n,g_0,g_0)}{\int_{S^n} \mu^{4/(n-2)}dV_{g_0}}.
	\end{align*}
	If \eqref{equality holds} holds with equality, then $1/\mu$ is an eigenfunction with eigenvalue $\lambda_0(S^n, \widetilde{g})$. By \eqref{laplace under conformal change}, we have
	\begin{eqnarray*}
		\lambda_0(S^n,g_0)=(-\Delta_{g_0}+c_nR_{g_0})1=\mu^{(n+2)/(n-2)}\, (-\Delta_{\widetilde g}+c_nR_{\widetilde g})(1/\mu)\\
		=\lambda_0(S^n,\widetilde{g})\mu^{4/(n-2)}.
	\end{eqnarray*}
	Therefore, $\mu$ is a constant function and the result follows.
\end{proof}

\textbf{Acknowledgements.}
The second author thanks Professor Hamid Hezari and Professor Zhiqin Lu for many useful discussions on this topic. The second author would also like to thank Professor Bernard Shiffman for his constant support and mentoring.

\bibliographystyle{plain} 
\bibliography{references}

\end{document}